\documentclass[12pt]{amsart}
\usepackage{latexsym,enumerate}
\usepackage{amssymb, xcolor,mathrsfs}
\usepackage{amsmath,amsthm,amsfonts,amssymb,latexsym, mathabx}
\usepackage{booktabs}
\usepackage{makecell}

\usepackage{listings}
\usepackage{xcolor}

\lstdefinelanguage{GAP}{
	morekeywords={if, then, elif, else, fi, for, while, do, od, function, local, return, end, and, or, not, in},
	sensitive=true,
	morecomment=[l]{\#},
	morestring=[b]",
	morestring=[b]',
}

\usepackage{comment}
\usepackage[
backend=biber,
style=numeric,      
sorting=nyt,        
maxnames=10,
giveninits=true,    
uniquename=false,
doi=false,
isbn=false,
url=true,
]{biblatex}
\newtheorem{theorem}{Theorem}[section]
\newtheorem{lemma}[theorem]{Lemma}
\newtheorem{proposition}[theorem]{Proposition}
\newtheorem{corollary}[theorem]{Corollary}

\newtheorem{conjecture}[theorem]{Conjecture}

\newtheorem{question}[theorem]{Question}

\theoremstyle{definition}
\newtheorem{definition}[theorem]{Definition}

\newtheorem*{thmA}{Theorem A}

\newcommand{\Q}{\mathbb{Q}}

\newcommand{\C}{\mathbb{C}}

\newcommand{\la}{\langle}
\newcommand{\ra}{\rangle}

\newcommand{\ol}{\overline}

\DeclareMathOperator{\SL}{SL}

\DeclareMathOperator{\Gal}{Gal}

\DeclareMathOperator{\Syl}{Syl}

\DeclareMathOperator{\Irr}{Irr}

\newcommand{\Cent}{\mathbf{Z}}

\makeatletter
\newcommand{\hathat}[1]{%
	\begingroup%
	\let\macc@kerna\z@%
	\let\macc@kernb\z@%
	\let\macc@nucleus\@empty%
	\hat{\raisebox{.35ex}{\vphantom{\ensuremath{#1}}}\smash{\hat{#1}}}%
	\endgroup%
}
\makeatother

\DeclareCiteCommand{\tabcite}
{\usebibmacro{cite:init}%
	\usebibmacro{prenote}}
{\usebibmacro{citeindex}%
	\usebibmacro{cite:comp}}
{}
{\usebibmacro{cite:dump}%
	\usebibmacro{postnote}}

\DefineBibliographyStrings{english}{%
	bibliography = {BIBLIOGRAPHY},
}

\usepackage{setspace}
\AtBeginBibliography{\singlespacing}
\DeclareFieldFormat{labelnumber}{#1}

\DeclareFieldFormat[article]{year}{(#1)}
\DeclareFieldFormat[book]{year}{#1}

\DeclareFieldFormat[article,book,inproceedings,thesis,unpublished,misc]{title}{#1}

\DeclareFieldFormat[article,inproceedings]{volume}{\mkbibbold{#1}}
\DeclareFieldFormat{journaltitle}{\mkbibemph{#1}}
\DeclareFieldFormat[book]{title}{\mkbibemph{#1}}
\DeclareFieldFormat{pages}{#1}
\DeclareFieldFormat{doi}{#1}
\DeclareFieldFormat[online]{title}{#1}
\DeclareFieldFormat{howpublished}{#1}
\DeclareFieldFormat{eprint}{preprint, arXiv:#1}
\DeclareFieldFormat{archivePrefix}{#1}
\DeclareFieldFormat[inproceedings]{booktitle}{#1}
\DeclareFieldFormat[inproceedings]{series}{\mkbibemph{#1}}

\renewbibmacro{in:}{}

\DeclareBibliographyDriver{article}{%
	\printnames{author}%
	\newunit\newblock
	\printfield{title}%
	\newunit\newblock
	\printfield{journaltitle}%
	\setunit{\addspace}%
	\printfield{volume}%
	\iffieldundef{number}{}{%
		(\printfield{number})%
	}%
	\setunit{\addcomma\addspace}%
	\iffieldundef{pages}{\addcomma}{
		\printfield{pages}%
	}
	\printfield{year}%
	\finentry
}

\DeclareBibliographyDriver{book}{%
	\printnames{author}%
	\newunit\newblock
	\printfield{title}%
	\iffieldundef{edition}{}{%
		\setunit{\addcomma\space}%
		\printfield[edition]{edition}%
		,
	}%
	\finentry
	\setunit{\space}
	\newunit\newblock
	\printlist{publisher}%
	\setunit{\addcomma\space}%
	\printfield{year}%
	.
	
}

\DeclareBibliographyDriver{misc}{%
	\printnames{author}%
	\newunit\newblock
	\printfield{title}%
	\setunit{\addcomma\space}%
	\printfield{eprint}
	\addcomma
	\addspace
	\printfield{year}%
	\finentry
}

\DeclareBibliographyDriver{unpublished}{%
	\printnames{author}%
	\newunit\newblock
	\printfield{title}%
	\iffieldundef{year}
	{}
	{%
		\setunit{\addcomma\space}%
		\printfield{year}%
	}%
	\finentry
}

\DeclareBibliographyDriver{inproceedings}{%
	\printnames{author}%
	\newunit\newblock
	\printfield{title}%
	\newunit\newblock
	\printfield{booktitle}%
	\newunit\newblock
	\printfield{series}%
	\setunit*{\addcomma\space}%
	\printfield{volume}%
	\newunit\newblock
	\printlist{publisher}%
	\setunit*{\addcomma\space}%
	\printlist{location}
	\setunit*{\addcomma\space}%
	\printfield{year}%
	\setunit{\addcomma\space}%
	\iffieldundef{pages}{}{%
		\printfield{pages}%
	}%
	\finentry
}

\DeclareBibliographyDriver{online}{%
	\printnames{author}%
	\newunit\newblock
	\printfield{title}%
	\newunit\newblock
	\printfield{howpublished}
	\setunit{\addcomma\space}
	\printfield{year}%
	\newunit\newblock
	\printfield{url}
	\finentry
}

\begin{document}
	
	\title[On a question of Navarro]{On a question of Navarro on the field of values of characters of solvable groups}

	\author[Christopher Herbig]{Christopher Herbig}
	\address{%
		Little Priest Tribal College\\
		Winnebago, NE 68071 \\
		USA}
	\email{christopher.herbig@littlepriest.edu}

	

	\date{\today}
	
	\maketitle
	
	\begin{abstract}
		In a 2023 survey paper, Gabriel Navarro posed the following problem: Given an irreducible character $\chi$ of some solvable group $G$ where $\chi$ either is 2-rational or has odd degree, does there exist some $g \in G$ such that $\Q(\chi(g)) = \Q(\chi)$? The answer was recently shown to be no in general by Ulrich Thiel. However, when $\chi$ is further assumed to be factorizable as a product of $p$-special characters, we are able to show that the answer is yes.
	\end{abstract}
	
	
\section{Introduction}

We first establish some notation and terminology. All groups are assumed to be finite, and all class functions of any group are assumed to take values in $\C$. Given a group $G$ and a class function $\chi$, we define $\Q(\chi)$ to be the field extension of $\Q$ generated by all values of $\chi$. We will use $\zeta_n$ to denote the primitive $n$th root of unity $\exp(2\pi i / n)$, and $\Q_n$ will denote the extension of $\Q$ generated by $\zeta_n$. Given any cyclotomic integer $\alpha$, we define $c(\alpha)$ to be the unique smallest integer such that $\Q(\alpha) \subseteq \Q_{c(\alpha)}$. We call $c(\alpha)$ the \emph{conductor} of $\alpha$. The conductor of a class function $\chi$, denoted $c(\chi)$, is defined analogously. If $n$ is an integer and $\pi$ is a set of primes, then $n_{\pi}$ is taken to be the largest divisor of $n$ whose prime divisors all lie in $\pi$. If $\pi = \{p\}$ for some prime $p$, then we will simply write $n_p$ instead of $n_{\{p\}}$. Given a set of primes $\pi$, we let $\pi'$ denote complement of $\pi$ among the set of all primes. Likewise, if $\pi = \{p\}$ for some prime $p$, we will simply write $p'$ instead of $\{p\}'$. For a group $G$, we let $\pi(G)$ denote the set of prime divisors of $|G|$. Otherwise, our notation will follow that of Isaacs in \cite{Is}. 

Gabriel Navarro has posed the following in \cite[Problem 7.3]{Nav}: If $G$ is a solvable group and $\chi \in \Irr(G)$ either is 2-rational or has odd degree, does there exist $g \in G$ such that $\Q(\chi(g)) = \Q(\chi)$? Using artificial intelligence, Ulrich Thiel has shown the answer to be no in \cite{Thiel} even when $G$ is assumed to have odd order. In GAP \cite{Gap}, one can construct the direct product of \texttt{SmallGroup(189,5)} with the cyclic group of order 9, and one sees that the condition fails for some of the degree 3 irreducible characters of this group. However, none of the characters of the solvable groups he provides as examples are quasiprimitive. This led us to wonder whether the answer is yes when $\chi$ is further assumed to be quasiprimitive. In fact, we are able to show that the answer is yes when $\chi$ is merely assumed to be fully factorizable as a product of $p$-special characters.\footnote{For a proof of the fact that quasiprimitive characters of solvable groups are fully factorizable, see Theorem 2.9 in \cite{Is2}.}
\begin{thmA}\label{big-conj}
	Let $G$ be a solvable group, and let $\chi \in \Irr(G)$ be a fully factorizable character which either is 2-rational or has odd degree. There exists $g \in G$ such that $\Q(\chi(g)) = \Q(\chi)$. 
\end{thmA}
In particular, the conclusion of this theorem holds for all fully factorizable characters of odd-order groups. It appears though that our assumption that either $\chi(1)$ is odd or $\chi$ is 2-rational cannot be readily relaxed as there seem to be numerous 2-groups possessing irreducible characters serving as counterexamples with the smallest such group having order 32.

We also remark that there are examples of primitive characters of nonsolvable groups for which the conclusion of Theorem A fails to hold, as demonstrated by Thiel. For instance, one can consider $Co_2 \times C_4$ where $Co_2$ is the second Conway group and $C_4$ is the cyclic group of order four. We take $\lambda$ to be a generator of $\Irr(C_4)$ and $\psi \in \Irr(Co_2)$ to be a character of degree 10,395 corresponding to either $\chi_{22}$ or $\chi_{23}$ in GAP \cite{Gap}. One can verify that $\psi \times \lambda \in \Irr(Co_2 \times C_4)$ is primitive and that the conclusion of Theorem A fails to hold for this character. Indeed, $\Q(\psi \times \lambda) = \Q(\sqrt{7},i)$, which has degree 4 over $\Q$. Yet, the field generated by any irrational value of $\psi \times \lambda$ is one of either $\Q(i)$, $\Q(\sqrt{7})$, or $\Q(i\sqrt{7})$, which each have degree 2 over $\Q$. 

Nonetheless, the conclusion of Theorem A is known to hold for the irreducible characters of certain families of nonsolvable groups by the work of the author and N. N. Hung, who showed in \cite{HH} that the conclusion holds for the Suzuki 2-groups and the groups $\SL(2,q)$ where $q$ is a prime power. One can also verify that the conclusion holds for the irreducible characters of each of the sporadic simple groups.  

Our proof for Theorem A is inspired by the proof of the solvable case of Feit's Conjecture provided in \cite[Theorem 2.22]{Is2}. We state Feit's conjecture \cite{F80} for reference:

\begin{conjecture}
	Let $G$ be a group, and let $\chi \in \Irr(G)$. There exists $g \in G$ having order $c(\chi)$.
\end{conjecture}

The connection between Navarro's problem and the Feit conjecture is not coincidental. Specifically, whenever the conclusion of Theorem A holds for an irreducible character $\chi$ of some group $G$, then Feit's conjecture will hold for that character as a result. To see this, assume there exists $g \in G$ such that $\Q(\chi(g)) = \Q(\chi)$. Clearly, $c(\chi(g)) = c(\chi)$. Now, if $\mathfrak{X}$ is a representation affording $\chi$, then the order of $g$ is the least common multiple of the orders of the eigenvalues of $\mathfrak{X}(g)$. Since $\chi(g)$ lies in the field generated by these eigenvalues, it follows that $\Q_{c(\chi)} = \Q_{c(\chi(g))} \subseteq \Q_{o(g)}$. We conclude that $c(\chi)$ divides the order of $g$, and the claim follows.

\section{Preliminary Lemmas}

Given a prime $p$ and two integers $a$ and $b$ where $a < b$, we have the following containment: $\Q_{p^a}\subseteq \Q_{p^b}$. In particular, the cyclotomic fields of the form $\Q_{p^a}$ are linearly ordered by containment. Consequently, given a character $\chi$ of a group $G$ having a prime-power conductor, there exists an element $g \in G$ such that the conductor of $\chi(g)$ equals the conductor of $\chi$ itself. We state this as a lemma for future use.
\begin{lemma}\label{L:2:1}
	Let $\chi$ be a character of a group $G$, and assume $c(\chi) = p^a$ for some prime $p$ an integer $a$. There exists $g \in G$ such that $c(\chi(g)) = c(\chi)$.
\end{lemma}

We will now review some basic facts about $p$-special characters. Proofs for each of the stated facts can be found in Chapter 2 of \cite{Is2}. Let $\chi$ be a fully factorizable character of a solvable group $G$, and let $\pi := \pi(G)$ be the set of prime divisors of $|G|$. For each $p \in \pi(G)$, we write the factorization of $\chi$ as
$$
\chi = \prod_{p \in \pi(G)} \chi_p,
$$
where each $\chi_p \in \Irr(G)$ is $p$-special. The $\chi_p$ satisfy many nice properties. For each $p$, we have that $c(\chi_p) = p^a$ for some $a$. By Gajendragadkar's restriction theorem, the restriction of $\chi_p$ to any Sylow $p$-subgroup is irreducible, and in fact, restriction to a Sylow $p$-subgroup $P$ defines an injective map from the set of $p$-special characters of $G$ into $\Irr(P)$. The above factorization of $\chi$ is unique up to the order of the factors. The field of values of $\chi$ can also be nicely expressed in terms of the fields generated by its $p$-special factors. The lemma below is taken from Corollary 2.13 and Lemma 2.19 in \cite{Is2}, and we state this result here since we shall refer to it frequently.
\begin{lemma}\label{L:Is-stuff}
	Let $G$ be a group, and let $\chi \in \Irr(G)$ be fully factorizable with factorization
	$$
	\chi = \prod_{p} \chi_p,
	$$
	where $\chi_p \in \Irr(G)$ is $p$-special. We have
	$$
	\Q(\chi) = \Q(\{\chi_p(g) \; | \; g \in G, \; p \in \pi(G)\}),
	$$
	and if $Q$ is a Sylow $q$-subgroup of $G$, then
	$$
	\Q(\chi_q) = \Q( (\chi_q)_Q ).
	$$
\end{lemma}

We are now able to prove the following proposition.

\begin{proposition}\label{P:2:1}
	Let $\psi \in \Irr(G)$ be $p$-special for an odd prime $p$. Then, $\Q(\psi) = \Q_{c(\psi)}$. In particular, if $\chi$ is a fully factorizable character of a solvable group and $\chi$ either has odd degree or is 2-rational, then $\Q(\chi) = \Q_{c(\chi)}$.
\end{proposition}
\begin{proof}
	Let $\psi \in \Irr(G)$ be $p$-special for an odd prime $p$. By replacing $G$ with $G/\ker\psi$, it is of no loss to assume that $\psi$ is faithful. We have that $c(\psi)$ is an integer power of $p$, say $p^a$. Of course, if $\psi$ is rational valued, there is nothing to prove, so we may assume otherwise. 
	
	Now, $|\Q_{c(\psi)} : \Q| = (p-1)p^{a-1}$. There exists $g \in G$ such that $p^{a-1} \mid |\Q(\chi(g)) : \Q|$ by Lemma \ref{L:2:1}. On the other hand, let $P$ be a Sylow $p$-subgroup of $G$. Since $\psi_P$ is irreducible, we have $\Cent(P) \subseteq \Cent(\psi) = \Cent(G)$. Thus, there exists $h \in G$ such that $\psi(h) = \zeta_p\psi(1)$. Now, one verifies that $\Q(\psi) \supseteq \Q(\psi(g),\psi(h)) = \Q_{c(\psi)}$, and so $\Q(\psi) = \Q_{c(\psi)}$.
	
	For the second claim, let $\chi$ be a fully factorizable character of $G$, and assume that $\chi(1)$ is odd or that $\chi$ is 2-rational. We factorize $\chi$ as a product of $p$-special characters:
	$$
	\chi = \prod_{p \in \pi(G)} \chi_p.
	$$
	By Lemma \ref{L:Is-stuff}, if $\chi$ is 2-rational, then $\chi_2$ is rational valued. Otherwise, if $\chi$ has odd degree, then $\chi_2$ is linear. In either case, $\Q(\chi_2) = \Q_{2^a}$ for some nonnegative integer $a$. In particular, the values of any $\chi_p$ generate a full cyclotomic field, and the result holds by Lemma \ref{L:Is-stuff}.
\end{proof}

We remark that the first statement in the above proposition fails to hold when $p = 2$. Taking the quaternion group of order 16 for instance, the two nonlinear, irrational characters each generate the subfield of $\Q_8$ fixed by complex conjugation. Also, for any irreducible character $\chi$ of an arbitrary solvable group, Cram has given a very short proof in \cite{Cram} that $|\Q_{c(\chi)} : \Q(\chi)|$ divides $\chi(1)$. For an irreducible character $\chi$ of a general finite group, it was conjectured in \cite{HT} that $|\Q_{c(\chi)} : \Q(\chi)| \leq \chi(1)$, and this conjecture remains open.

To state our next lemma, we  require the notion of a ``magic field automorphism'' of $\Gal(\Q_n)$, which Isaacs defines in Section 2C of \cite{Is2} as a certain type of automorphism of $\Q_n$ for an odd integer $n$. For our purposes though, it will be  convenient to extend this definition to even integers. First, recall that for an integer $n$ and a set of primes $\pi$, we may write $\Q_n = \Q_{n_\pi}\Q_{n_{\pi'}}$ where $\Q_{n_\pi} \cap \Q_{n_{\pi'}} = \Q$. Thus, defining a Galois automorphism on $\Q_{n_\pi}$ and $\Q_{n_{\pi'}}$ amounts to defining a Galois automorphism on $\Q_n$.
\begin{definition}
	Let $n$ be an integer, and let $\pi$ be a set of primes. We may write $\Q_n = \Q_{n_\pi}\Q_{n_{\pi'}}$. A \emph{magic field automorphism} $\tau_\pi \in \Gal(\Q_n)$ is the automorphism defined by letting $\tau_\pi$ act as the trivial automorphism on $\Q_{n_{\pi'}}$ and as complex conjugation on $\Q_{n_{\pi - \{2\}}}$. Further, if $2 \in \pi$ and $4 \mid n$, then $\tau_\pi$ acts on $\Q_{n_2}$ by sending $\zeta_{n_2}$ to $-\zeta_{n_2}$. If $\pi = \{p\}$ for some prime $p$, then we simply write $\tau_{\{p\}}$ as $\tau_p$. 
\end{definition} 
If $n$ is odd or divisible by 4, then it is clear that the subgroup of magic field automorphisms in $\Gal(\Q_n)$ is an elementary abelian 2-group whose rank is the number of prime divisors of $n$.

We introduce yet more notation. Let $n$ be an integer with prime factorization $n = \prod_{p \mid n} p^{a_p}$ for $a_p \geq 1$. The Galois group $\Gal(\Q_n)$ decomposes as a direct product: $$\Gal(\Q_n) = \prod_{p \mid n} \Gal(\Q_{p^{a_p}}).$$ Specifically, given $\sigma \in \Gal(\Q_n)$, we may uniquely write $\sigma= \prod_{p \mid n}\sigma_p$, where $\sigma_p$ fixes $\Q_{q^{a_q}}$ for $q \neq p$.

\begin{lemma}\label{L:2:3}
	Let $n$ be an integer with prime factorization $n = \prod_{p \mid n} p^{a_p}$ for $a_p \geq 1$, and for each $p$ dividing $n$, let $\alpha_p \in \Q_{p^{a_p}}$.
	Further assume that if $n$ is even, then $\alpha_2$ is a root of unity. Set $\alpha = \prod_{p \mid n} \alpha_p$, and assume there exists a nonidentity Galois automorphism $\sigma \in \Gal(\Q_n/\Q(\alpha))$. Either of the following occur:
	\begin{itemize}
		\item[(a).] $\alpha$ is fixed by some $\sigma_p$. Further, if $n$ is even and $\alpha$ is fixed by $\sigma_2$, then we may assume $\sigma_2$ is the automorphism of $\Gal(\Q_{2^{a_2}})$ sending $\zeta_{2^{a_2}}$ to $-
		\zeta_{2^{a_2}}$.
		\item[(b).] $\sigma$ is a magic field automorphism with respect to some set $\pi$. Here, $|\pi|$ is even, and for each $p \in \pi$, we have $\alpha_p^{\tau_\pi} = -\alpha_p$. 
	\end{itemize}
	
\end{lemma}
\begin{proof}
	Keeping the above notation, if $\alpha_p$ is fixed by any nontrivial Galois automorphism of $\Q_{p^{a_p}}$, say $\tau$, then $\tau$ extends to an automorphism of $\Q_n$ which leaves $\Q_{q^{a_q}}$ fixed for $q \neq p$. Thus, $\tau$ fixes $\alpha$, and we can assume from now on that $\Q(\alpha_p) = \Q_{p^{a_p}}$ for each $p$ dividing $n$. Also, notice that in the case where $p = 2$ here, since $\alpha_2$ is a root of unity, the fact that $\tau$ fixes $\alpha_2$ implies that $c(\alpha_2) < 2^{a_2}$, meaning that the Galois automorphism of $\Gal(\Q_{2^{a_2}})$ sending $\zeta_{2^{a_2}}$ to $-
	\zeta_{2^{a_2}}$ must fix $\alpha_2$. 
	
	If $\sigma \in \Gal(\Q_n/\Q(\alpha))$ is a nonidentity automorphism, we have
	$$
	\prod_{p \mid n} \alpha_p^{\sigma_p} = \alpha^\sigma =  \alpha = \prod_{p\mid n} \alpha_p,
	$$
	and so
	\begin{equation}\label{rat-prod}  
		1 = \prod_{p \mid n} \dfrac{\alpha_p^{\sigma_p}}{\alpha_p}.
	\end{equation} 
	Now, it is straightforward to verify using the
	Galois Correspondence that each quotient in the above
	product must be rational. Thus, for each $p$,
	there exists a rational number $r_p$ such that $\alpha_p^{\sigma_p} = r_p\alpha_p$.  If
	$\sigma_p$ has order $\ell$, then $\alpha_p = \alpha_p^{\sigma_p^\ell} = r_p^\ell\alpha_p$. We conclude that $r_p$ is a rational
	root of unity and is therefore either 1 or $-1$. 
	
	Now, since we are assuming $\sigma$ is a nonidentity automorphism, at least one $\sigma_p$ is a nonidentity automorphism. If there exists a $p$ for which $\sigma_p$ is nontrivial and $r_p = 1$, then $\alpha$ is fixed by $\sigma_p$. Thus, from now on, we may assume that for each $p$ either $\sigma_p$ is trivial or $r_p = -1$. 
	
	If $p$ is odd, we claim that if $r_p = -1$, then $\sigma_p$ acts as complex conjugation on $\Q_{p^{a_p}}$. Indeed, $\alpha_p^{\sigma_p} = -\alpha_p$, and so $\sigma_p$ acts as an involution on $\Q(\alpha_p) = \Q_{p^{a_p}}$. When $p$ is odd, $\Gal(\Q_{p^{a_p}})$ has one involution, namely complex conjugation. Thus, for each odd $p$ where $r_p = -1$, $\sigma$ acts as complex conjugation on $\Gal(\Q_{p^{a_p}})$, and $\sigma$ acts trivially for all $p$ where $r_p = 1$. We conclude that $\sigma$ is a magic field automorphism with respect to $\pi$ where $\pi$ is the set of all $p$ such that $r_p = -1$. The fact that $|\pi|$ is even now follows from the fact that $r_p$ takes values in $\{-1,1\}$ and the fact that \eqref{rat-prod} becomes
	$$
	1 = \prod_{p \mid n} \dfrac{\alpha_p^{\sigma_p}}{\alpha_p} =  \prod_{p \mid n} \dfrac{r_p\alpha_p}{\alpha_p} = \prod_{p \mid n} r_p.
	$$
\end{proof}

\section{Proof of the Main Theorem}

Let $n$ be an integer, and let $\sigma \in \Gal(\Q_n)$. We define the $\Q$-linear function $f_{\sigma}: \Q_n \to \Q_n$ by setting $f_\sigma(\alpha) = \alpha - \alpha^\sigma$. Since the Galois group is abelian, it is straightforward to see that the functions are commutative under composition, i.e., $f_\sigma \circ f_\tau = f_\tau \circ f_\sigma$ for $\sigma, \tau \in \Gal(\Q_n)$. 

Also, let $X \subseteq \Gal(\Q_n)$, and let $f_X$ denote the composition of all $f_\sigma$ for $\sigma \in X$. The commutativity of the $f_\sigma$ under composition implies that $f_X$ is well defined.  As a convention, if $X$ is empty, we take $f_X$ to be the identity automorphism. We will require the following facts, which we state as a lemma.
\begin{lemma}
	Let $n$ be an integer, and let $X \subseteq \Gal(\Q_n)$. The following hold:
	\begin{itemize}
		\item[(a).] If $\alpha$ lies in a subfield of $\Q_n$ fixed by any element of $X$, then $f_X(\alpha) = 0$.
		\item[(b).] For any $\alpha \in \Q_n$, we have
		$$
		f_X(\alpha) = \sum_{i=0}^{|X|} (-1)^i\sum_{A \in \mathscr{X}_i}\alpha^{\prod_{\sigma \in A}\sigma}, 
		$$
		where $\mathscr{X}_i$ is the set of all size-$i$ subsets of $X$. (Vacuous products are of course taken to be the identity.)
	\end{itemize}
\end{lemma}
\begin{proof}
	For (a), if $\sigma \in X$ fixes $\alpha \in \Q_n$, then we may write 
	$$
	f_X(\alpha) = f_{X - \{\sigma\}}(f_\sigma(\alpha) ) = f_{X - \{\sigma\}}(0) = 0,
	$$
	from which the claim follows.
	
	The formula in part (b) follows easily from induction on the cardinality of $X$.
\end{proof}

We are now ready to prove Theorem A.

\begin{proof}[Proof of Theorem A]
	Let $G$ be solvable, and let $\chi \in \Irr(G)$ be fully factorizable, and assume either that $\chi(1)$ is odd or that $\chi$ is 2-rational. By assumption, $\chi$ is fully factorizable with factorization
	$$
	\chi = \prod_{p \in \pi} \chi_{p}
	$$
	for some set of primes $\pi$ and where $\chi_{p}$ is $p$-special for each prime $p \in \pi$. 
	
	Furthermore, in the case where $\chi$ is 2-rational, we have that $\chi_2$ is a rational-valued character by Lemma 2.2, and so, $\Q(\chi) = \Q(\prod_{p \neq 2} \chi_p)$, so we may assume that $\chi_2$ is the principal character. Then, the restriction of $\chi$ to a Hall $2'$-subgroup $H$ is irreducible by Gajendragadkar's restriction theorem, and by Lemma 2.2, we have that $\Q(\chi) = \Q(\chi_H)$. Replacing $G$ with $H$ and $\chi$ with $\chi_H$, it follows that in order to prove Theorem A in the case where $\chi$ is 2-rational, we need only prove the statement when $G$ is assumed to have odd order. 
	
	By Proposition \ref{P:2:1}, we have that $\Q(\chi) = \Q_n$ for some integer $n$ having prime factorization $n = \prod_{p \in \pi} p^{a_p}$ where $a_p \geq 1$ for each $p$. In particular, for any two distinct $\sigma, \tau \in \Gal(\Q_n)$, we have that $\chi^{\sigma} \neq \chi^\tau$.
	
	Now, $\Gal(\Q_n)$ decomposes as a direct product like so:
	$$
	\Gal(\Q_n) = \prod_{p \in \pi} \Gal(\Q_{p^{a_p}}).
	$$
	For odd primes $p$, we have $\Gal(\Q_{p^{a_p}}) \cong C_{p-1} \times C_{p^{a_p - 1}}$. For each odd prime $p \in \pi$ and each prime divisor $q$ of $\varphi(p^{a_p}) = (p-1)p^{a_p-1}$ choose an element $\sigma_{p,q}$ of order $q$ from $\Gal(\Q_{p^{a_p}})$. If $4 \mid n$, let $\sigma_2$ denote the automorphism of $\Q_{2^{a_2}}$ sending $\zeta_{2^{a_2}}$ to $-\zeta_{2^{a_2}}$ Let $X$ denote the set consisting of $\sigma_2$ (if $4 \mid n$) and all $\sigma_{p,q}$ for each odd $p \in \pi$ and $q$ dividing $(p-1)p^{a_p - 1}$. Since $\sigma_2$ and the $\sigma_{p,q}$ each lie in a unique direct factor of $\Gal(\Q_n)$, it follows that all automorphisms of the form
	$$
	\prod_{\sigma \in A} \sigma 
	$$
	for $A \subseteq X$ are distinct. (Again, vacuous products are taken to be the identity). Now, defining $f_X$ to be the composition of all $f_\sigma$ for $\sigma \in X$, it follows from the lemma above that
	\begin{equation}\label{E:expan}
		f_X(\chi) = \sum_{i=0}^{|X|} (-1)^i\sum_{A \in \mathscr{X}_i}\chi^{\prod_{\sigma \in A}\sigma}.
	\end{equation}
	Each term in this sum is a distinct irreducible character and the generalized character is therefore nonzero. 
	
	By Lemma \ref{L:2:3}, if $\chi(g)$ is fixed by some 
	Galois automorphism which is not a magic field automorphism for some $g \in G$, then $\chi(g)$ is 
	fixed by some Galois automorphism of the form $\sigma_p$ for $p \in \pi$ which leaves $\Q_{q^{a_q}}$ 
	fixed for all $q \neq p$. From our definition of $X$, it follows that $\chi(g)$ is fixed by some element of $X$, and we conclude that
	$f_X(\chi(g)) = 0$. Thus, $f_X(\chi)$ vanishes on all elements of $G$ except for the $g \in G$ where $\Gal(\Q_n/\Q(\chi(g)))$ 
	consists of magic field automorphisms with respect to sets having even cardinality.

	Assume that there exists no $g \in G$ such that $\Q(\chi(g)) = \Q(\chi)$. We now claim that $f_X(\chi(g))/4$ is an algebraic integer for each $g \in G$. 
	Let $g \in G$ be an element for which $f_X(\chi(g)) \neq 0$ so that $\chi(g)$ is fixed by a nontrivial magic field automorphism $\tau_\pi$ where $|\pi| \geq 2$. If $q \in \pi$, then $\chi_q(g)^{\tau_\pi} = -\chi_q(g)$. (Notice that this choice of $q$ depends upon $g$.) It follows that for distinct $q,r \in \pi$, we have
	$$
	f_{\tau_q}(f_{\tau_r}(\chi(g))) = 4\chi(g).
	$$ 
	One sees that by our construction of $X$, we have $\tau_q, \tau_r \in X$. Setting $Y = X - \{\tau_q, \tau_r\}$, we have
	\begin{equation}\label{E:fourinator}
		f_X(\chi(g)) = f_Y (f_{\tau_q}(f_{\tau_r}(\chi(g)))  ) = f_Y(4\chi(g)) = 4f_Y(\chi(g)).
	\end{equation}
	Set $\alpha_{g} = f_Y(\chi(g))$, and notice that $\alpha_g$ is an algebraic integer. 
	
	Now, assume that $\chi$ is 2-rational. By our remarks in the second paragraph of this proof, we can assume that $|G|$ is odd.
	By the orthogonality relations and the distinctness of the terms in \eqref{E:expan}, we have that
	$$
	\sum_{g \in G} f_X(\chi(g))\ol{\chi(g)}  = |G|[f_X(\chi),\chi] = |G|.
	$$
	On the other hand, we have that
	$$
	\sum_{g \in G} f_X(\chi(g))\ol{\chi(g)} = 4\sum_{\substack{g \in G, \\ f_X(\chi(g)) \neq 0}} \alpha_g \ol{\chi(g)}.
	$$
	Combining the above two equations, we obtain that $|G|/4$ is an algebraic integer, contradicting the assumption that $|G|$ is odd.
	
	Now, assume that $\chi$ has odd degree. We may assume now that the 2-special factor of $\chi$ is a nonprincipal linear character. Since we assume that $\chi$ is faithful, this implies that the restriction of $\chi_2$ to a Sylow 2-subgroup of $G$ is faithful, and it follows that the Sylow 2-subgroups of $G$ are cyclic and that $N := \ker\chi_2$ is a normal 2-complement in $G$. Notice also that $\chi_N$ is irreducible by Gajendragadkar's restriction theorem.  
	
	Now, choose a coset $xN$ of $G$ such that $\la xN \ra = G/N$. We seek to verify the three identities below:\footnote{We remark that these claims are analogous to Lemma 8.6 in \cite{Is2}.}
	\begin{equation}\label{E:identity-case}
		\sum_{g \in xN} \chi(g)\ol{\chi(g)} = |N|,
	\end{equation}
	\begin{equation}\label{E:zero-case}
		\sum_{g \in xN} \chi^\sigma(g)\ol{\chi(g)} = 0 \mbox{ for } \sigma \in \Gal(\Q_n) \mbox{ restricting nontrivially to } \Q_{n_{2'}}, \mbox{ and}
	\end{equation} 
	\begin{equation}\label{E:weird-case}
		\sum_{g \in xN} \chi^{\tau_2}(g)\ol{\chi(g)} = -|N|.
	\end{equation}
	We handle these three identities simultaneously. For $\sigma \in \Gal(\Q_n)$, define a class function $\Xi(hN)$ on $G/N$ by setting 
	$$\Xi(hN) = \dfrac{1}{|N|} \sum_{g \in hN} \chi^\sigma(g)\ol{\chi(g)}.$$
	We wish to evaluate the multiplicities of this class function's irreducible constituents. Let $\lambda \in \Irr(G/N)$. We compute
	$$
	[\Xi,\lambda] = \dfrac{1}{|G:N|} \sum_{hN \in G/N} \ol{\lambda(hN)} \dfrac{1}{|N|}\sum_{g \in hN} \chi^\sigma(g)\ol{\chi(g)} = \dfrac{1}{|G|}\sum_{g \in G} \chi^{\sigma}(g) \ol{\chi(g)\lambda(g)} = [\chi^\sigma,\chi\lambda].
	$$ 
	Thus, as $\chi^\sigma$ and $\chi\lambda$ are both irreducible, it follows that $[\Xi,\lambda] \neq 0$ iff $\chi^\sigma = \chi\lambda$. 
	
	If $\sigma$ is the identity automorphism, then since $\chi_N$ is irreducible, it follows from the Gallagher Correspondence that $\chi = \chi\lambda$ iff $\lambda$ is the principal character. We conclude that $\Xi$ is the principal character  of $G/N$, and \eqref{E:identity-case} now follows.
	
	Now, if the restriction of $\sigma$ to $\Q_{n_{2'}}$ is nontrivial, then by Lemma 2.2 and Proposition \ref{P:2:1}, we have that $\chi_N^\sigma \neq \chi_N = (\chi\lambda)_N$ for each $\lambda \in \Irr(G/N)$, and thus, $\Xi$ is identically zero in this case, from which the equality in \eqref{E:zero-case} follows. 
	
	Now, we consider the case where $\sigma = \tau_2$. We claim that $\chi$ is nonzero on a Sylow 2-subgroup $P$ of $G$. Indeed, if $\chi(z) = 0$ for some $z \in P$, then this would imply $\chi(1)$ is even by Lemma 2.24 in \cite{Is2}. Now, if $\chi^\sigma = \chi\lambda$, then $\lambda_P = (\chi^\sigma)_P / (\chi)_P = (\chi_2^\sigma)_P (\ol{\chi_2})_P$. Now, $\sigma$ sends primitive $|P|$th roots of unity their respective negatives and fixes the $|P|/2$th roots of unity. Thus, $\lambda_P$ (and consequently $\lambda$ itself) is the sign character. We conclude that $\lambda = \Xi$. Since we assume $xN$ is a generator of $G/N$, it follows that $\Xi(xN) = -1$, and the equality in \eqref{E:weird-case} follows.
	
	We now compute
	$$
	\sum_{g \in xN} f_X(\chi(g))\ol{\chi(g)} = \sum_{g \in xN}\bigg(\sum_{i=0}^{|X|} (-1)^i\sum_{A \in \mathscr{X}_i}\chi^{\prod_{\sigma \in A}\sigma}(g)\bigg)\ol{\chi(g)} $$ $$= \sum_{i=0}^{|X|} (-1)^i\sum_{A \in \mathscr{X}_i}\bigg(\sum_{g \in xN} \chi^{\prod_{\sigma \in A}\sigma}(g) \ol{\chi(g)}\bigg).
	$$
	By our construction of $X$, $\prod_{\sigma \in A}\sigma$ restricts nontrivially to $\Q_{n_{2'}}$ unless $A \subseteq \{\tau_2\}$. Thus, by \eqref{E:identity-case}, \eqref{E:zero-case}, and \eqref{E:weird-case}, the above sum becomes
	$$
	\sum_{g \in xN} f_X(\chi(g))\ol{\chi(g)} = \sum_{g \in xN} \chi(g)\ol{\chi(g)} - \sum_{g \in xN} \chi(g)^{\tau_2}\ol{\chi(g)} = |N| - (-|N|) = 2|N|.
	$$
	However, by \eqref{E:fourinator}, we have that 
	$$
	\sum_{g \in xN} f_X(\chi(g))\ol{\chi(g)} = 4\sum_{\substack{g \in xNG, \\ f_X(\chi(g)) \neq 0}} \alpha_g\ol{\chi(g)}.
	$$
	Combining the above two equations, we obtain that $(2|N|)/4$ is an algebraic integer, contradicting the fact that $|N|$ is odd. 
\end{proof}

\section{Concluding Remarks}

A sort of intermediate question to Navarro's problem and the Feit conjecture is the following question, which was considered in \cite{HH} and earlier in Section 9 of \cite{BKNT}:

\begin{question}\label{other-question}
	Let $G$ be a group, and let $\chi \in \Irr(G)$. When does there exist $g \in G$ such that $c(\chi(g)) = c(\chi)$.
\end{question}

It is straightforward to verify that for any $\chi$ where the conclusion of Theorem A holds, then the answer to this question is yes. Likewise, whenever the answer to this question is yes for some $\chi$, then Feit's conjecture holds for $\chi$. The answer is known to be no in general for imprimitive characters of solvable groups. In particular, \texttt{SmallGroup(960,730)} in GAP \cite{Gap} is a supersolvable group possessing an irreducible character for which the answer is no. However, the answer is yes for primitive characters of solvable groups. In fact, the proof for the solvable case of the Feit conjecture provided in \cite[Theorem 2.22]{Is2} generalizes readily to this question when $\chi$ is assumed to be primitive (or even just fully factorizable):

\begin{theorem}
	Let $G$ be solvable, and let $\chi \in \Irr(G)$ be fully factorizable as a product of $p$-special characters. There exists $g \in G$ such that $c(\chi(g)) = c(\chi)$.
\end{theorem}
\begin{proof}
	Let $\chi$ be as above and set $n = c(\chi)$. Let $\chi$ have factorization
	$$
	\chi = \prod_p \chi_p 
	$$
	where each $\chi_p$ is $p$-special.
	By Lemma \ref{L:Is-stuff}, we have that $c(\chi_p) = n_p$ for each $p$. Thus, for any prime divisor of $q$ not equal to $p$, we have that $\Q(\chi_p) \subseteq \Q_{n/q}$ or equivalently that $c(\chi_p)$ divides $n/q$.
	
	Letting $q$ be a prime divisor of $n$, and notice that $\Q_{n/q} \subset \Q(\zeta_{n/q}, \chi) \subseteq \Q_n$. We may choose a Galois automorphism $\omega_q \in \Gal(\Q_n/\Q_{n/q})$ which does not leave $\Q(\zeta_{n/q}, \chi)$ fixed. In particular, $\omega_q$ acts nontrivially on the values of $\chi$, and $\chi^{\omega_q} \neq \chi$.
	
	Now, for distinct primes $p$ and $q$ dividing $n$, we have that $\omega_q$ fixes $\chi_p$ since $\Q(\chi_p) \subseteq \Q_{n/q}$. Thus, since the factorization of $\chi$ into $p$-special characters is unique up to the order of the factors, we conclude that $\chi_q^{\omega_q} \neq \chi_q$.  
	
	If we assume that there exists no $g \in G$ such that $c(\chi(g)) = n$, then for every $x \in G$,  $c(\chi(x))$ properly divides $n$, and there exists a prime $q$ such that $c(\chi(x))$ divides $n/q$. For such an $x$, it follows that for each $p$, we have $\chi_p(x) \subseteq \Q_{n/q}$, and in particular, $\chi_q(x) \in \Q_{n/q}$. We conclude that $\chi_q(x)^{\omega_q} = \chi_q(x)$.	Thus, the below product is identically zero:
	$$
	0= \prod_q (\chi_q - \chi_q^{\omega_q}).
	$$
	Let $\pi(n)$ denote the set of prime divisors of $n.$ We may rewrite this product as
	$$
	0= \sum_{A \subseteq \pi(n)} (-1)^{|A|}\prod_{q} \chi_q^{\sigma_q}
	$$
	where $\sigma_q = \omega_q$ if $q \in A$ and $\sigma_q = 1$ otherwise. Now, one notices that each product of $p$-special characters in the above sum is a distinct irreducible character since fully factorizable characters factor uniquely. Thus, the sum cannot be zero by the linear independence of irreducible characters, a contradiction.
\end{proof}

\subsection*{Disclosure of computational assistance} The author wishes to clarify that artificial intelligence was only used for the purpose of typesetting the bibliography of this manuscript. 

\printbibliography[heading=bibintoc]
	
\end{document}